\documentclass[11pt,a4paper]{article}

\usepackage{inputenc}
\usepackage{amsmath}
\usepackage{bm}
\usepackage{bbold}
\usepackage{amsthm}
\usepackage{enumerate}

\usepackage{hyperref}

\title{Using Tropical Optimization Techniques to\\ Evaluate Alternatives via Pairwise Comparisons\thanks{This work was supported in part by the Russian Foundation for Humanities (grant No. 16-02-00059).}}

\author{N. Krivulin\thanks{St.~Petersburg State University, Universitetskaya nab.~7/9, St.~Petersburg, 199034, Russia, nkk@math.spbu.ru.}}

\date{}

\newtheorem{theorem}{Theorem}
\newtheorem{lemma}[theorem]{Lemma}
\newtheorem{corollary}[theorem]{Corollary}

\theoremstyle{definition}
\newtheorem{example}{Example}

\begin{document}

\maketitle

\begin{abstract}
We describe a new approach based on tropical optimization techniques to solve the problem of rating alternatives from pairwise comparison data. The problem is formulated to approximate, in the log-Chebyshev sense, pairwise comparison matrices by reciprocal matrices of unit rank, and then represented in general terms of tropical mathematics as a tropical optimization problem. The optimization problem takes a common, unified form for both multiplicative and additive comparison scales. We apply recent results in tropical optimization to offer new complete solutions to the rating problems under various assumptions about the pairwise comparison matrices. The solutions are given in a compact vector form, which extends known solutions and involves modest computational efforts. The results obtained are illustrated with numerical examples. Specifically, we show by example that the partial solution known before may miss better results provided by the new complete solution. An example to demonstrate a tropical analogue of the analytical hierarchy process decision scheme is also given.
\\

\textbf{Key-Words:} idempotent semifield, tropical optimization, matrix approximation, pairwise comparison, consistent matrix.
\\

\textbf{MSC (2010):} 65K10, 15A80, 65K05, 41A50, 90B50
\end{abstract}

\section{Introduction}

Tropical optimization problems present an important research and application domain of tropical mathematics, which finds use in solving real-world problems in various fields, including project scheduling, location analysis and decision making. Tropical (idempotent) mathematics deals with the theory and application of semirings with idempotent addition, and dates back to a few influential works by Pandit \cite{Pandit1961Anew}, Cuninghame-Green \cite{Cuninghamegreen1962Describing}, Giffler \cite{Giffler1963Scheduling}, Hoffman \cite{Hoffman1963Onabstract}, Vorob{'}ev \cite{Vorobjev1963Theextremal}, and Romanovski\u{\i} \cite{Romanovskii1964Asymptotic}, appeared in the early 1960s.  

In later decades, a large body of literature was published on various aspects of tropical mathematics, including recent monographs by Kolokoltsov and Maslov \cite{Kolokoltsov1997Idempotent}, Golan \cite{Golan2003Semirings}, Heidergott et al. \cite{Heidergott2006Maxplus}, Itenberg et. al \cite{Itenberg2007Tropical}, Gondran and Minoux \cite{Gondran2008Graphs}, McEneaney \cite{Mceneaney2010Maxplus}, and Maclagan and Sturmfels \cite{Maclagan2015Introduction}, as well as numerous research and applied papers. Many studies, such as the early papers \cite{Cuninghamegreen1962Describing,Hoffman1963Onabstract}, were motivated and illustrated by extremal problems drawn from operations research, computer science, and other areas. A noticeable part of these problems can be directly formulated and solved within the framework of tropical mathematics as tropical optimization problems.

Tropical optimization problems are usually formulated to minimize or maximize a function defined on vectors over an idempotent semifield (a semiring with multiplicative inverses), with or without constraints on the feasible solutions. The objective function can be either linear or nonlinear; the nonlinear objective functions are typically defined through multiplicative conjugate transposition of vectors. The constraints take the form of linear vector equations and inequalities. Some of the problems are completely solved under quite general assumptions in a closed compact vector form, whereas the other problems have solutions available only in the form of a numerical algorithm, which produces a partial solution if any exists, or indicates the infeasibility of the problem. For further details and references, one can consult, for instance, the overviews by Krivulin \cite{Krivulin2014Tropical,Krivulin2015Amultidimensional}.

One of the application areas, where tropical optimization can be of service, is the analysis of preferences based on pairwise comparison data in decision making. Given a matrix that contains the result of pairwise comparisons of alternatives, the problem of interest is to rate (judge, score) the alternatives. The problem arises in various contexts (see the classical works by Thurstone \cite{Thurstone1927Alaw}, Saaty \cite{Saaty1980Theanalytic}, and David \cite{David1988Method} to list a few). In the tropical mathematics setting, the problem is examined in the papers by Elsner and van den Driessche \cite{Elsner2004Maxalgebra,Elsner2010Maxalgebra}, Tran \cite{Tran2013Pairwise}, and Gursoy et al. \cite{Gursoy2013Theanalytic}, where a solution is proposed to calculate a tropical eigenvector of the matrix as the vector of final scores assigned to the alternatives. The solution is used in \cite{Gursoy2013Theanalytic} as an alternative calculation technique for the Analytical Hierarchy Process (AHP) decision scheme \cite{Saaty1980Theanalytic}. 

New general methods of tropical optimization were recently developed by Krivulin in \cite{Krivulin2014Aconstrained,Krivulin2015Extremal,Krivulin2015Amultidimensional} to provide a useful framework for solving many optimization problems, including the above problem of rating alternatives in decision making \cite{Krivulin2015Rating}. The methods offer direct solutions in a closed vector form, which is suitable for further analysis and practical application.

The purpose of this paper is to describe a new approach based on tropical optimization techniques to solve the problem of rating alternatives from pairwise comparisons. We formulate the problem to approximate, in log-Chebyshev sense, pairwise comparison matrices by reciprocal matrices of unit rank, and then represent the problem in general terms of tropical mathematics as a tropical optimization problem. The optimization problem takes a common, unified form for both multiplicative and additive comparison scales.

We apply recent results in tropical optimization to offer new complete solutions to the rating problems under various assumptions about the pairwise comparison matrices. The solutions are given in a compact vector form, which extends known solutions and involves modest computational efforts. The results obtained are illustrated with numerical examples. We show by example that the partial solution, which was known before, may miss better results provided by the new complete solution. An example to demonstrate a tropical analogue of the AHP decision scheme is also given.

The paper is organized as follows. Section~\ref{S-PCJ} outlines the method of ranking alternatives from their pairwise comparisons as a starting point for the subsequent discussion. Furthermore, in Section~\ref{S-BDNR}, we offer a brief overview of preliminary definitions, notation and results of tropical mathematics to provide an analytical framework for the development of new solutions. In Section~\ref{S-MAP}, the problems of approximating square matrices by reciprocal matrices of unit rank are solved as tropical optimization problems. Section~\ref{S-ACM} uses results of the previous section to derive the vector of final scores from pairwise comparison data. Finally, Section~\ref{S-AES} demonstrates various examples of solving problems to evaluate the scores of alternatives.

\section{Pairwise Comparison Judgments}
\label{S-PCJ}

The method of pairwise comparisons is widely accepted in decision making to estimate preferences when a direct evaluation of the preferences is impossible or infeasible. Given a set of alternatives, the method uses the result of pairwise comparisons with an appropriate scale to make judgment on the relative preference of each alternative by evaluating its individual score or rating (see, e.g., \cite{Thurstone1927Alaw,Saaty1980Theanalytic,David1988Method,Saaty2013Onthemeasurement} for further details).

\subsection{Pairwise Comparison Matrices}
\label{S-PCM}

The results of comparing the alternatives in pairs on a multiplicative or additive scale are described by pairwise comparison matrices that have a specific form. Consider a pairwise comparison matrix $\bm{A}=(a_{ij})$ obtained with a multiplicative scale. The entry $a_{ij}$ of the matrix shows that alternative $i$ is preferred to alternative $j$ by $a_{ij}$ times. The multiplicative comparison matrix is symmetrically reciprocal, which means that it has only positive entries satisfying the condition
$$
a_{ij}
=
1/a_{ji}.
$$

If an additive scale is used, the matrix $\bm{A}$ has the skew-symmetric form with the entries, which compare preferences in terms of differences, and fit the equality
$$
a_{ij}
=
-a_{ji}.
$$

In practice, however, multiplicative (additive) pairwise comparison matrices can have a form that differs from the symmetrically reciprocal (skew-symmetric) form, due to inexact or incorrect measurements.

To provide consistency of and to avoid confusion in the data given by pairwise comparison matrices, these data must be transitive, which requires that the entries of the multiplicative (additive) comparison matrix satisfy the equality
$$
a_{ij}
=
a_{ik}a_{kj}
\qquad
(a_{ij}
=
a_{ik}+a_{kj}).
$$ 

If a pairwise comparison matrix has only transitive entries, it is called consistent. For each multiplicative (additive) consistent matrix $\bm{A}=(a_{ij})$, there is a positive (real) vector $\bm{x}=(x_{i})$ whose elements completely determine the entries of $\bm{A}$ as follows:
$$
a_{ij}
=
x_{i}/x_{j}
\qquad
(a_{ij}
=
x_{i}-x_{j}).
$$

Finally, provided that a matrix $\bm{A}$ is consistent, its corresponding vector $\bm{x}$ is considered to represent directly (up to a positive factor) the individual scores of alternatives, and thus offers a solution to the problem of analysis of preference in question.

\subsection{Approximation by Consistent Matrices}

The pairwise comparison matrices encountered in real-world applications are generally inconsistent, which can be caused by various reasons from limitations in human judgment to errors in the source data. Therefore, the problem of approximating a pairwise comparison matrix $\bm{A}$ by a consistent matrix arises in the general form
\begin{equation}
\begin{aligned}
&
\text{minimize}
&&
d(\bm{A},\bm{X}),
\end{aligned}
\label{P-dAX}
\end{equation}
where the minimum is taken over all consistent matrices $\bm{X}$, and $d$ is a suitable measure of approximation error.

Note that any consistent matrix $\bm{X}$ is uniquely determined, up to a positive factor, by a single vector $\bm{x}$, and hence the solution of problem \eqref{P-dAX} is equivalent to finding $\bm{x}$. Since, in the context of the analysis of preference, the vector $\bm{x}$ presents the overall individual scores of alternatives, the problem of evaluating the scores is reduced to the above approximation problem.

To solve problem \eqref{P-dAX} for a pairwise comparison matrix, several approaches are used, which include approximation with the principal eigenvector of the matrix \cite{Saaty1980Theanalytic,Saaty1984Comparison}, least squares approximation \cite{Saaty1984Comparison,Chu1998Ontheoptimal} and other approximation techniques \cite{Barzilai1997Deriving,Farkas2003Consistency,Gonzalezpachon2003Transitive,Gavalec2015Decision}. These approaches normally provide algorithmic solutions in the form of numerical procedures, such as the power iterations in the principal eigenvector method, and the Newton algorithm in the least squares approximation.

Another solution approach, based on the application of tropical mathematics, is proposed and investigated in \cite{Elsner2004Maxalgebra,Elsner2010Maxalgebra,Tran2013Pairwise,Gursoy2013Theanalytic}. This approach consists of approximating the pairwise comparison matrix by a consistent matrix defined by a tropical eigenvector, and hence constitutes a tropical counterpart of the conventional principal eigenvector method. Moreover, the analysis in \cite{Elsner2010Maxalgebra} shows that the approximate consistent matrices, which solve the problem in the tropical mathematics setting, can be provided not only by tropical eigenvectors, but by other vectors as well. As one of the results of \cite{Elsner2010Maxalgebra}, a technique to find new solutions is proposed, which, however, offers a computational algorithm, rather than gives a direct solution in an explicit form.

In subsequent sections, we formulate the problem of finding approximate consistent matrices as a general problem of approximation by reciprocal matrices of rank one in the topical mathematics setting. It is shown that the problem can be considered as matrix approximation in the Chebyshev or Chebyshev-like metrics. We apply recent results in tropical optimization to offer complete, direct solutions of the approximation problems.

\section{Basic Definitions, Notation and Results}
\label{S-BDNR}

In this section, we follow the presentation given in \cite{Krivulin2014Aconstrained,Krivulin2014Tropical,Krivulin2015Extremal,Krivulin2015Amultidimensional} to outline preliminary definitions and results of tropical mathematics, which offer an analytical framework to the solutions in the subsequent sections. For further details and discussion, one can consult recent publications \cite{Kolokoltsov1997Idempotent,Golan2003Semirings,Heidergott2006Maxplus,Itenberg2007Tropical,Gondran2008Graphs,Mceneaney2010Maxplus,Maclagan2015Introduction}.

\subsection{Idempotent Semifield}

Let $\mathbb{X}$ be a carrier set that is equipped with binary operations $\oplus$ and $\otimes$, called addition and multiplication, which have neutral elements $\mathbb{0}$ and $\mathbb{1}$, called the zero and identity. Both operations are associative and commutative, and multiplication is distributive over addition. Addition is idempotent, resulting in the equality $x\oplus x=x$ for all $x\in\mathbb{X}$. Moreover, multiplication is invertible, implying that each nonzero $x\in\mathbb{X}$ has its inverse $x^{-1}$ such that $x\otimes x^{-1}=\mathbb{1}$. Together with these operations, the carrier set $\mathbb{X}$ forms the algebraic system $(\mathbb{X},\oplus,\otimes,\mathbb{0},\mathbb{1})$, which is usually referred to as the idempotent semifield.

The semifield is considered linearly ordered by an order that is consistent with the partial order produced by idempotent addition to define $x\leq y$ if and only if $x\oplus y=y$. In addition, the semifield is assumed to be algebraically complete, which means that the equation $x^{p}=a$, where $x^{p}$ denotes the iterated product, has solutions for any $a\in\mathbb{X}$ and integer $p>0$, and thus the powers with rational exponents are well-defined.

In the algebraic expressions below, the multiplication sign $\otimes$, as usual, is omitted for the sake of brevity. The power notation is always understood in the sense of tropical mathematics.

Examples of the idempotent semifield under consideration include
\begin{align*}
\mathbb{R}_{\max,\times}
&=
(\mathbb{R}_{+}\cup\{0\},\max,\times,0,1),
\\
\mathbb{R}_{\max,+}
&=
(\mathbb{R}\cup\{-\infty\},\max,+,-\infty,0),
\end{align*}
where $\mathbb{R}$ is the set of reals, and $\mathbb{R}_{+}=\{x\in\mathbb{R}|x>0\}$.

The semifield $\mathbb{R}_{\max,\times}$ is equipped with the addition $\oplus$ defined as maximum, and the multiplication $\otimes$ defined as usual. The neutral elements $\mathbb{0}$ and $\mathbb{1}$ coincide with the arithmetic zero and one, respectively. The power and inversion notations have the usual meaning. 

The semifield $\mathbb{R}_{\max,+}$ involves $\oplus=\max$, $\otimes=+$, $\mathbb{0}=-\infty$ and $\mathbb{1}=0$. For each $x\in\mathbb{R}$, the inverse $x^{-1}$ coincides with the conventional opposite number $-x$. The power notation $x^{y}$, where $x,y\in\mathbb{R}$, corresponds to the regular arithmetic product $xy$.

In both semifields, the order induced by idempotent addition corresponds to the natural linear order on $\mathbb{R}$.

\subsection{Vector and Matrix Algebra}

The column vectors with $n$ elements over $\mathbb{X}$ form the set $\mathbb{X}^{n}$. A vector with all elements equal to $\mathbb{0}$ is the zero vector denoted $\bm{0}$. A vector is called regular if it has no zero elements.

Vector addition and scalar multiplication follow the usual element-wise rules, where the scalar operations $\oplus$ and $\otimes$ act as the standard addition and multiplication.

A vector $\bm{b}$ is said to be linearly dependent on vectors $\bm{a}_{1},\ldots,\bm{a}_{m}$ if there are scalars $x_{1},\ldots,x_{m}$ such that $\bm{b}=x_{1}\bm{a}_{1}\oplus\cdots\oplus x_{m}\bm{a}_{m}$. Vectors $\bm{a}$ and $\bm{b}$ are collinear if $\bm{b}=x\bm{a}$ for some $x$.

The system of vectors $\bm{a}_{1},\ldots,\bm{a}_{m}$ is called linearly dependent if at least one of the vectors is dependent on others, and independent otherwise. The set of linear combinations $x_{1}\bm{a}_{1}\oplus\cdots\oplus x_{m}\bm{a}_{m}$ for all $x_{1},\ldots,x_{m}$ is closed under vector addition and scalar multiplication, and is referred to as the idempotent vector space generated by the system.

The multiplicative conjugate transpose of a nonzero column vector $\bm{x}=(x_{i})$ is a row vector $\bm{x}^{-}=(x_{i}^{-})$ with the elements $x_{i}^{-}=x_{i}^{-1}$ if $x_{i}\ne\mathbb{0}$, and $x_{i}^{-}=\mathbb{0}$ otherwise.

The set of matrices with $m$ rows and $n$ columns is denoted by $\mathbb{X}^{m\times n}$. A matrix with all zero entries is the zero matrix. Matrix addition, matrix multiplication and scalar multiplication are given by the conventional entry-wise formulae, where the operations $\oplus$ and $\otimes$ play the role of the usual addition and multiplication.

Consider a matrix $\bm{A}\in\mathbb{X}^{m\times n}$. The transpose of $\bm{A}$ is the matrix denoted by $\bm{A}^{T}\in\mathbb{X}^{n\times m}$. The multiplicative conjugate transpose of any nonzero matrix $\bm{A}=(a_{ij})$ is the matrix $\bm{A}^{-}=(a_{ij}^{-})$ with the entries $a_{ij}^{-}=a_{ji}^{-1}$ if $a_{ji}\ne\mathbb{0}$, and $a_{ij}^{-}=\mathbb{0}$ otherwise.

The rank of a matrix is defined as the maximum number of linearly independent columns (rows) in the matrix. A matrix $\bm{A}$ has rank $1$ if and only if $\bm{A}=\bm{x}\bm{y}^{T}$, where $\bm{x}$ and $\bm{y}$ are regular column vectors. 

Consider the square matrices of order $n$ in the set $\mathbb{X}^{n\times n}$. A matrix that has $\mathbb{1}$ along the diagonal and $\mathbb{0}$ elsewhere is the identity matrix denoted $\bm{I}$. The power notation indicates iterated products as $\bm{A}^{0}=\bm{I}$ and $\bm{A}^{p}=\bm{A}^{p-1}\bm{A}$ for any matrix $\bm{A}$ and integer $p>0$.

A matrix $\bm{A}$ that satisfies the condition $\bm{A}^{-}=\bm{A}$ is called symmetrically reciprocal (or, simply, reciprocal). A reciprocal matrix $\bm{A}$ is of unit rank if and only if $\bm{A}=\bm{x}\bm{x}^{-}$, where $\bm{x}$ is a regular column vector.

The trace of a matrix $\bm{A}=(a_{ij})\in\mathbb{X}^{n\times n}$ is given by
$$
\mathop\mathrm{tr}\bm{A}
=
a_{11}\oplus\cdots\oplus a_{nn}.
$$

For any matrices $\bm{A}$ and $\bm{B}$, and scalar $x$, the following equalities hold:
\begin{gather*}
\mathop\mathrm{tr}(\bm{A}\oplus\bm{B})
=
\mathop\mathrm{tr}\bm{A}
\oplus
\mathop\mathrm{tr}\bm{B},
\qquad
\mathop\mathrm{tr}(\bm{A}\bm{B})
=
\mathop\mathrm{tr}(\bm{B}\bm{A}),
\\
\mathop\mathrm{tr}(x\bm{A})
=
x\mathop\mathrm{tr}\bm{A}.
\end{gather*}

\subsection{Distance Functions}

The distance between two regular vectors $\bm{x},\bm{y}\in\mathbb{X}^{n}$ is defined by the function
$$
d(\bm{x},\bm{y})
=
\bm{y}^{-}\bm{x}
\oplus
\bm{x}^{-}\bm{y},
$$
which attains the minimum value $\mathbb{1}$ when $\bm{y}=\bm{x}$. 

In the semifield $\mathbb{R}_{\max,+}$, this function has the minimum $\mathbb{1}=0$ and coincides with the Chebyshev metric
$$
d_{\infty}(\bm{x},\bm{y})
=
\max_{1\leq i\leq n}|x_{i}-y_{i}|.
$$

For $\mathbb{R}_{\max,\times}$, the function $d$ becomes a Chebyshev distance in the log-log scale after taking the logarithm,
$$
d_{\infty}^{\prime}(\bm{x},\bm{y})
=
\log d(\bm{x},\bm{y})
=
\max_{1\leq i\leq n}|\log x_{i}-\log y_{i}|.
$$

In the general case, the function $d$ is called the Chebyshev-like distance.

For two matrices $\bm{A},\bm{B}\in\mathbb{X}^{n\times n}$ without zero entries, the Chebyshev-like distance function is given by
\begin{equation}
d(\bm{A},\bm{B})
=
\mathop\mathrm{tr}(\bm{B}^{-}\bm{A})
\oplus
\mathop\mathrm{tr}(\bm{A}^{-}\bm{B}),
\label{E-dAB}
\end{equation}
which takes the form of the Chebyshev metric in the semifield $\mathbb{R}_{\max,+}$, and of a log-Chebyshev distance after logarithmic transformation in $\mathbb{R}_{\max,\times}$.

\subsection{Eigenvalues and Eigenvectors of Matrices}

A scalar $\lambda\in\mathbb{X}$ is an eigenvalue of a matrix $\bm{A}\in\mathbb{X}^{n\times n}$ if there exists a nonzero vector $\bm{x}\in\mathbb{X}^{n}$ such that the equality $\bm{A}\bm{x}=\lambda\bm{x}$ holds. The vector $\bm{x}$, which satisfies the equality, is an eigenvector of $\bm{A}$, corresponding to $\lambda$.

The maximum eigenvalue of a matrix $\bm{A}=(a_{ij})$ is called the spectral radius of $\bm{A}$, and calculated as
$$
\lambda
=
\mathop\mathrm{tr}\nolimits\bm{A}\oplus\cdots\oplus\mathop\mathrm{tr}\nolimits^{1/n}(\bm{A}^{n}),
$$
or, in terms of the entries of the matrix $\bm{A}$, as
\begin{equation}
\lambda
=
\bigoplus_{k=1}^{n}\bigoplus_{1\leq i_{1},\ldots,i_{k}\leq n}(a_{i_{1}i_{2}}a_{i_{2}i_{3}}\cdots a_{i_{k}i_{1}})^{1/k}.
\label{E-lambda-ai1i2ai2i3aiki1}
\end{equation}

Any matrix $\bm{A}$ with nonzero entries has only one eigenvalue, which coincides with the spectral radius $\lambda$ given by the above expressions. The eigenvectors of $\bm{A}$, which correspond to $\lambda$, are derived as follows. Calculate the matrix $\bm{A}_{\lambda}=\lambda^{-1}\bm{A}$, and then find the matrices
$$
\bm{A}_{\lambda}^{\ast}
=
\bm{I}\oplus\bm{A}_{\lambda}\oplus\cdots\oplus\bm{A}_{\lambda}^{n-1},
$$
and $\bm{A}_{\lambda}^{\times}=\bm{A}_{\lambda}\bm{A}_{\lambda}^{\ast}$. Finally, the matrix $\bm{A}_{\lambda}^{+}$ is formed by taking those columns that coincide in both matrices $\bm{A}_{\lambda}^{\ast}$ and $\bm{A}_{\lambda}^{\times}$. All eigenvectors of the matrix $\bm{A}$ are given by
$$
\bm{x}
=
\bm{A}_{\lambda}^{+}\bm{u},
$$
where $\bm{u}$ is any nonzero vector, and hence constitute an idempotent vector space spanned by the columns in $\bm{A}_{\lambda}^{+}$.

\subsection{Tropical Optimization Problem}

Consider the following optimization problem in the tropical mathematics setting: given a matrix $\bm{A}\in\mathbb{X}^{n\times n}$, find regular vectors $\bm{x}=(x_{j})\in\mathbb{X}^{n}$ that
\begin{equation}
\begin{aligned}
&
\text{minimize}
&&
\bm{x}^{-}\bm{A}\bm{x},
\\
&
\text{subject to}
&&
x_{j}
>
\mathbb{0},
\quad
j=1,\ldots,n.
\end{aligned}
\label{P-minxAx}
\end{equation}

The next complete, direct solution to the problem is obtained using various arguments in \cite{Krivulin2014Aconstrained,Krivulin2015Extremal,Krivulin2015Amultidimensional}.
\begin{lemma}
\label{L-minxAx}
Let $\bm{A}$ be a matrix with spectral radius $\lambda>\mathbb{0}$, and $\bm{A}_{\lambda}=\lambda^{-1}\bm{A}$. Then, the minimum value in problem \eqref{P-minxAx} is equal to $\lambda$, and all regular solutions are given by
$$
\bm{x}
=
\bm{A}_{\lambda}^{\ast}\bm{u},
\qquad
\bm{u}
\ne
\bm{0}.
$$
\end{lemma}

From Lemma~\ref{L-minxAx} it follows that the solutions (together with the zero vector) form an idempotent vector space generated by the columns of the matrix $\bm{A}_{\lambda}^{\ast}$.

\section{Matrix Approximation Problems}
\label{S-MAP}

In this section, we examine problems of approximating square matrices by reciprocal matrices of unit rank. The problems are formulated in terms of tropical mathematics, and then solved as tropical optimization problems.

Let $\bm{A}\in\mathbb{X}^{n\times n}$ be a matrix. Consider the approximation problem to find regular vectors $\bm{x}=(x_{j})\in\mathbb{X}^{n}$ that
\begin{equation}
\begin{aligned}
&
\text{minimize}
&&
d(\bm{A},\bm{x}\bm{x}^{-}),
\\
&
\text{subject to}
&&
x_{j}
>
\mathbb{0},
\quad
j=1,\ldots,n;
\end{aligned}
\label{P-dAxx}
\end{equation}
where $d$ is the Chebyshev-like distance function \eqref{E-dAB}, which is taken as a measure of approximation error.

\subsection{Approximation of One Matrix}

We start with approximating a matrix that may not be reciprocal.

\begin{theorem}
\label{T-mindAxx}
Let $\bm{A}$ be a matrix such that the matrix $\bm{B}=\bm{A}\oplus\bm{A}^{-}$ has no zero entries, $\mu$ be the spectral radius of $\bm{B}$, and $\bm{B}_{\mu}=\mu^{-1}\bm{B}$. Then, the minimum value in problem \eqref{P-dAxx} is equal to $\mu$, and all solutions are given by
$$
\bm{x}
=
\bm{B}_{\mu}^{\ast}\bm{u},
\qquad
\bm{u}
\ne
\bm{0}.
$$
\end{theorem}
\begin{proof}
Considering $(\bm{x}\bm{x}^{-})^{-}=\bm{x}\bm{x}^{-}$, we use properties of the trace to write the objective function in \eqref{P-dAxx} as
\begin{equation*}
d(\bm{A},\bm{x}\bm{x}^{-})
=
\mathop\mathrm{tr}((\bm{x}\bm{x}^{-})^{-}\bm{A})
\oplus
\mathop\mathrm{tr}(\bm{A}^{-}\bm{x}\bm{x}^{-})
=
\bm{x}^{-}\bm{A}\bm{x}
\oplus
\bm{x}^{-}\bm{A}^{-}\bm{x}
=
\bm{x}^{-}\bm{B}\bm{x}.
\end{equation*}

Since the matrix $\bm{B}$ has no zero entries, it follows from \eqref{E-lambda-ai1i2ai2i3aiki1} that $\bm{B}$ has the spectral radius $\mu\ne\mathbb{0}$. Therefore, we can define the matrix $\bm{B}_{\mu}=\mu^{-1}\bm{B}$, and then apply Lemma~\ref{L-minxAx} to obtain the solution.
\end{proof}

Suppose that the matrix $\bm{A}$ is reciprocal. Then, we have $\bm{A}^{-}=\bm{A}$, and the theorem reduces to the following result, which closely reproduces that of Lemma~\ref{L-minxAx}. 
\begin{corollary}
\label{C-mindAxx}
Let $\bm{A}$ be a reciprocal matrix with spectral radius $\lambda$, and $\bm{A}_{\lambda}=\lambda^{-1}\bm{A}$. Then, the minimum in \eqref{P-dAxx} is equal to $\lambda$, and all solutions are given by
$$
\bm{x}
=
\bm{A}_{\lambda}^{\ast}\bm{u},
\qquad
\bm{u}
\ne
\bm{0}.
$$
\end{corollary}

\subsection{Approximation of Several Matrices}

We now suppose that there are $m$ matrices $\bm{A}_{1},\ldots,\bm{A}_{m}\in\mathbb{X}^{n\times n}$ to determine a reciprocal matrix of rank $1$ that approximates these matrices simultaneously. The approximation problem can be formulated in the form
\begin{equation}
\begin{aligned}
&
\text{minimize}
&&
\max_{1\leq i\leq m}d(\bm{A}_{i},\bm{x}\bm{x}^{-}),
\\
&
\text{subject to}
&&
x_{j}
>
\mathbb{0},
\quad
j=1,\ldots,n.
\end{aligned}
\label{P-dmaxAixx}
\end{equation}

For nonreciprocal matrices, a solution is as follows. 
\begin{theorem}
\label{T-dmaxAixx}
Let $\bm{A}_{i}$ be matrices for $i=1,\ldots,m$, such that the matrix $\bm{B}=\bm{A}_{1}\oplus\bm{A}_{1}^{-}\oplus\cdots\oplus\bm{A}_{m}\oplus\bm{A}_{m}^{-}$ has no zero entries, $\mu$ be the spectral radius of $\bm{B}$, and $\bm{B}_{\mu}=\mu^{-1}\bm{B}$. Then, the minimum value in problem \eqref{P-dmaxAixx} is equal to $\mu$, and all solutions are given by
$$
\bm{x}
=
\bm{B}_{\mu}^{\ast}\bm{u},
\qquad
\bm{u}
\ne
\bm{0}.
$$
\end{theorem}
\begin{proof}
For each $i=1,\ldots,m$, the same argument as in Theorem~\ref{T-mindAxx} yields $d(\bm{A}_{i},\bm{x}\bm{x}^{-})=\bm{x}^{-}(\bm{A}_{i}\oplus\bm{A}_{i}^{-})\bm{x}$. We replace $\max$ by $\oplus$, and rewrite the objective function as
$$
\max_{1\leq i\leq m}d(\bm{A}_{i},\bm{x}\bm{x}^{-})
=
\bm{x}^{-}\bigoplus_{i=1}^{m}(\bm{A}_{i}\oplus\bm{A}_{i}^{-})\bm{x}
=
\bm{x}^{-}\bm{B}\bm{x}.
$$

An application of Lemma~\ref{L-minxAx} to the problem with the new objective function completes the proof.
\end{proof}

If the matrices $\bm{A}_{1},\ldots,\bm{A}_{m}$ are reciprocal, the statement of Theorem~\ref{T-dmaxAixx} is valid in the next reduced form.

\begin{corollary}
\label{C-dmaxAixx}
Let $\bm{A}_{i}$ be reciprocal matrices for $i=1,\ldots,m$, the matrix $\bm{B}=\bm{A}_{1}\oplus\cdots\oplus\bm{A}_{m}$ have spectral radius $\mu$, and $\bm{B}_{\mu}=\mu^{-1}\bm{B}$. Then, the minimum in \eqref{P-dmaxAixx} is equal to $\mu$, and all solutions are given by
$$
\bm{x}
=
\bm{B}_{\mu}^{\ast}\bm{u},
\qquad
\bm{u}
\ne
\bm{0}.
$$
\end{corollary}

\subsection{Weighted Approximation of Matrices}

Let $\bm{A}_{1},\ldots,\bm{A}_{m}\in\mathbb{X}^{n\times n}$ be matrices and $w_{1},\ldots,w_{m}\in\mathbb{X}$ be scalars. We consider these scalars as weights to write the problem of weighted approximation in the form
\begin{equation}
\begin{aligned}
&
\text{minimize}
&&
\max_{1\leq i\leq m}w_{i}d(\bm{A}_{i},\bm{x}\bm{x}^{-}),
\\
&
\text{subject to}
&&
x_{j}
>
\mathbb{0},
\quad
j=1,\ldots,n.
\end{aligned}
\label{P-widmaxAixx}
\end{equation}

The next result provides a solution to the problem in the general case of nonreciprocal matrices.
\begin{theorem}
\label{T-widmaxAixx}
Let $\bm{A}_{i}$ be matrices for $i=1,\ldots,m$, such that the matrix $\bm{B}=w_{1}(\bm{A}_{1}\oplus\bm{A}_{1}^{-})\oplus\cdots\oplus w_{m}(\bm{A}_{m}\oplus\bm{A}_{m}^{-})$ has no zero entries, $\mu$ be the spectral radius of $\bm{B}$, and $\bm{B}_{\mu}=\mu^{-1}\bm{B}$. Then, the minimum value in problem \eqref{P-widmaxAixx} is equal to $\mu$, and all solutions are given by
$$
\bm{x}
=
\bm{B}_{\mu}^{\ast}\bm{u},
\qquad
\bm{u}
\ne
\bm{0}.
$$
\end{theorem}
\begin{proof}
In the similar way as before, we first write
$$
\max_{1\leq i\leq m}w_{i}d(\bm{A}_{i},\bm{x}\bm{x}^{-})
=
\bm{x}^{-}\bigoplus_{i=1}^{m}w_{i}(\bm{A}_{i}\oplus\bm{A}_{i}^{-})\bm{x}
=
\bm{x}^{-}\bm{B}\bm{x},
$$
and then apply Lemma~\ref{L-minxAx} to complete the proof.
\end{proof}

We conclude with the case of reciprocal matrices.

\begin{corollary}
\label{C-widmaxAixx}
Let $\bm{A}_{i}$ be reciprocal matrices for $i=1,\ldots,m$, the matrix $\bm{B}=w_{1}\bm{A}_{1}\oplus\cdots\oplus w_{m}\bm{A}_{m}$ have spectral radius $\mu$, and $\bm{B}_{\mu}=\mu^{-1}\bm{B}$. Then, the minimum in \eqref{P-widmaxAixx} is equal to $\mu$, and all solutions are given by
$$
\bm{x}
=
\bm{B}_{\mu}^{\ast}\bm{u},
\qquad
\bm{u}
\ne
\bm{0}.
$$
\end{corollary}

\section{Approximation by Consistent Matrices}
\label{S-ACM}

In the framework of tropical mathematics, both multiplicative and additive consistent matrices $\bm{X}$ can be represented in a common form of the reciprocal matrix of rank $1$, given by the condition
$$
\bm{X}
=
\bm{x}\bm{x}^{-},
$$
to be interpreted in terms of the semifields $\mathbb{R}_{\max,\times}$ for the multiplicative case, and $\mathbb{R}_{\max,+}$ for the additive.

The problem of finding an approximate consistent matrix for pairwise comparison matrices can then be solved as an approximation problem \eqref{P-dAxx} in the log-Chebyshev or Chebyshev sense. Theorems~\ref{T-mindAxx}--\ref{T-widmaxAixx} and their corollaries provide new complete direct solutions to problem \eqref{P-dAxx} under various assumptions, and hence to the problem of evaluating the scores of alternatives from pairwise comparisons in the analysis of preference.

The application of these results offers a further improvement of the known solutions in \cite{Elsner2004Maxalgebra,Elsner2010Maxalgebra}. Specifically, the new solutions replace the computation of all the eigenvectors of a matrix $\bm{A}$ as the columns in the matrix $\bm{A}_{\lambda}^{+}$ obtained from the matrix $\bm{A}_{\lambda}^{\ast}$ to the calculation of $\bm{A}_{\lambda}^{\ast}$ alone. At the same time, these solutions can provide a significant extension of the solution set since the column space of $\bm{A}_{\lambda}^{\ast}$ is known to include the eigenspace for $\bm{A}$, which is generated by columns in $\bm{A}_{\lambda}^{+}$.

Note that it may appear that the eigenspace of $\bm{A}$ coincides with the column space of $\bm{A}_{\lambda}^{\ast}$ suggested by the new solutions, as in the following trivial example.
\begin{example}
{\rm
Consider the reciprocal matrix
$$
\bm{A}
=
\left(
\begin{array}{cc}
1 & a
\\
a^{-1} & 1
\end{array}
\right).
$$

Using formula \eqref{E-lambda-ai1i2ai2i3aiki1} in the context of the semifield $\mathbb{R}_{\max,\times}$ yields $\lambda=1$. Furthermore, we have
$$
\bm{A}_{\lambda}
=
\bm{A},
\qquad
\bm{A}_{\lambda}^{\ast}
=
\bm{I}\oplus\bm{A}_{\lambda}
=
\bm{A},
\qquad
\bm{A}_{\lambda}^{\times}
=
\bm{A}_{\lambda}\bm{A}_{\lambda}^{\ast}
=
\bm{A}.
$$

Since $\bm{A}_{\lambda}^{\times}=\bm{A}_{\lambda}^{\ast}$, we can take $\bm{A}_{\lambda}^{+}=\bm{A}_{\lambda}^{\ast}$, and thus conclude that the eigenspace of $\bm{A}$ (the column space of $\bm{A}_{\lambda}^{+}$) and the column space of $\bm{A}_{\lambda}^{\ast}$ coincide. 
}
\end{example}

The next example shows that, in general, the column space of $\bm{A}_{\lambda}^{\ast}$ is larger than the eigenspace of $\bm{A}$. Moreover, those columns of $\bm{A}_{\lambda}^{\ast}$ which do not belong to the eigenspace may offer better solutions to the problem of evaluating scores than the eigenvectors.

\begin{example} 
{\rm
Let us examine the pairwise comparisons on a multiplicative scale, given by the matrix 
$$
\bm{A}
=
\left(
\begin{array}{cccc}
1 & 2 & 1/2 & 1/2
\\
1/2 & 1 & 2 & 1/2
\\
2 & 1/2 & 1 & 1/2
\\
2 & 2 & 2 & 1
\end{array}
\right).
$$

A simple analysis of the pairwise comparison matrix yields the conclusion that the last alternative should be ranked first. The first three alternatives have lower ranks, but cannot be further differentiated.

For a formal analysis in the setting of the semifield $\mathbb{R}_{\max,\times}$, we apply \eqref{E-lambda-ai1i2ai2i3aiki1}, which gives
$$
\lambda
=
(a_{12}a_{23}a_{31})^{1/3}
=
2.
$$

Next, we define the matrix
$$
\bm{A}_{\lambda}
=
\lambda^{-1}\bm{A}
=
\left(
\begin{array}{cccc}
1/2 & 1 & 1/4 & 1/4
\\
1/4 & 1/2 & 1 & 1/4
\\
1 & 1/4 & 1/2 & 1/4
\\
1 & 1 & 1 & 1/2
\end{array}
\right).
$$

Furthermore, we calculate
\begin{equation*}
\bm{A}_{\lambda}^{2}
=
\left(
\begin{array}{cccc}
1/4 & 1/2 & 1 & 1/4
\\
1 & 1/4 & 1/2 & 1/4
\\
1/2 & 1 & 1/4 & 1/4
\\
1 & 1 & 1 & 1/4
\end{array}
\right),
\qquad
\bm{A}_{\lambda}^{3}
=
\left(
\begin{array}{cccc}
1 & 1/4 & 1/2 & 1/4
\\
1/2 & 1 & 1/4 & 1/4
\\
1/4 & 1/2 & 1 & 1/4
\\
1 & 1 & 1 & 1/4
\end{array}
\right).
\end{equation*}

Finally, we can obtain and compare the matrices
\begin{gather*}
\bm{A}_{\lambda}^{\ast}
=
\bm{I}\oplus\bm{A}_{\lambda}\oplus\bm{A}_{\lambda}^{2}\oplus\bm{A}_{\lambda}^{3}
=
\left(
\begin{array}{cccc}
1 & 1 & 1 & 1/4
\\
1 & 1 & 1 & 1/4
\\
1 & 1 & 1 & 1/4
\\
1 & 1 & 1 & 1
\end{array}
\right),
\\
\bm{A}_{\lambda}^{\times}
=
\bm{A}_{\lambda}\bm{A}_{\lambda}^{\ast}
=
\left(
\begin{array}{cccc}
1 & 1 & 1 & 1/4
\\
1 & 1 & 1 & 1/4
\\
1 & 1 & 1 & 1/4
\\
1 & 1 & 1 & 1/2
\end{array}
\right).
\end{gather*}

The first three identical columns in $\bm{A}_{\lambda}^{\ast}$ coincide with the same columns in $\bm{A}_{\lambda}^{\times}$, and thus present an eigenvector. However, in the context of decision making, this eigenvector gives no chance to rank alternatives. Although the last column in $\bm{A}_{\lambda}^{\ast}$ is not an eigenvector, it assigns the highest score to the last alternative, and equally lower scores for the others, which ranks the alternatives in line with the given matrix $\bm{A}$.
}
\end{example}

To conclude this section, we briefly comment on the computational complexity involved in the procedure of calculating the matrix $\bm{A}_{\lambda}^{\ast}$ from a pairwise comparison matrix $\bm{A}$ of order $n$. We note that the most computationally intensive part of the calculations is evaluating the first $n$ powers of the matrix $\bm{A}$. It is not difficult to see that these powers can be obtained with no more than $O(n^{4})$ scalar operations, which results at most in the same order of complexity for the entire procedure.

\section{Application to Evaluation of Scores}
\label{S-AES}

Below, we present examples that demonstrate the use of the tropical optimization technique to solve particular problems of finding the scores based on pairwise comparisons on a multiplicative scale. Considering that, in the general terms of tropical mathematics, the solutions have a common form for both multiplicative and additive scales, examples that assume an additive scale of comparison are omitted.

\subsection{Evaluation of Scores Given by One Matrix}

We start with an example of evaluating the score vector based on one reciprocal matrix. The case of a positive matrix that may not be reciprocal is presented next.   

\begin{example}
{\rm
\label{E-Areciprocal}
Consider a reciprocal matrix defined as
$$
\bm{A}
=
\left(
\begin{array}{cccc}
1 & 3 & 4 & 2
\\
1/3 & 1 & 1/2 & 1/3
\\
1/4 & 2 & 1 & 4
\\
1/2 & 3 & 1/4 & 1
\end{array}
\right).
$$

To approximate this matrix by a reciprocal matrix of unit rank, and thus to evaluate a score vector $\bm{x}$, we apply Corollary~\ref{C-mindAxx} in the setting of the semifield $\mathbb{R}_{\max,\times}$. First, we find the spectral radius $\lambda$ for $\bm{A}$. Application of \eqref{E-lambda-ai1i2ai2i3aiki1} gives
$$
\lambda
=
(a_{13}a_{34}a_{42}a_{21})^{1/4}
=
2.
$$

Furthermore, we consider the matrix
$$
\bm{A}_{\lambda}
=
\lambda^{-1}\bm{A}
=
\left(
\begin{array}{cccc}
1/2 & 3/2 & 2 & 1
\\
1/6 & 1/2 & 1/4 & 1/6
\\
1/8 & 1 & 1/2 & 2
\\
1/4 & 3/2 & 1/8 & 1/2
\end{array}
\right),
$$
and then calculate the matrices
\begin{equation*}
\bm{A}_{\lambda}^{2}
=
\left(
\begin{array}{cccc}
1/4 & 2 & 1 & 4
\\
1/12 & 1/4 & 1/3 & 1/2
\\
1/2 & 3 & 1/4 & 1 
\\
1/4 & 3/4 & 1/2 & 1/4 
\end{array}
\right),
\qquad
\bm{A}_{\lambda}^{3}
=
\left(
\begin{array}{cccc}
1 & 6 & 1/2 & 2
\\
1/8 & 3/4 & 1/6 & 2/3
\\
1/2 & 3/2 & 1 & 1/2
\\
1/8 & 1/2 & 1/2 & 1
\end{array}
\right).
\end{equation*}

Finally, we obtain the matrix
$$
\bm{A}_{\lambda}^{\ast}
=
\bm{I}\oplus\bm{A}_{\lambda}\oplus\bm{A}_{\lambda}^{2}\oplus\bm{A}_{\lambda}^{3}
=
\left(
\begin{array}{cccc}
1 & 6 & 2 & 4
\\
1/6 & 1 & 1/3 & 2/3
\\
1/2 & 3 & 1 & 2
\\
1/4 & 3/2 & 1/2 & 1
\end{array}
\right).
$$

Note that the columns in the matrix $\bm{A}_{\lambda}^{\ast}$ are collinear to each other. Specifically, the second, third and fourth columns can be obtained by multiplying the first one by $6$, $2$ and $4$, respectively. Since all columns generate exactly the same vector space, it is sufficient to use only one of them to represent all solution vectors. We take the first column and write the score vector as
$$
\bm{x}
=
\left(
\begin{array}{c}
1
\\
1/6
\\
1/2
\\
1/4
\end{array}
\right)u,
$$
where the factor $u$ can be arbitrary set to a positive number in accordance with the required form or desired interpretation of the result.

Assuming $u=1$, the vector $\bm{x}=(1,1/6,1/2,1/4)^{T}$ shows that the first alternative has the highest score $x_{1}=1$, followed by the third and fourth with scores $x_{3}=1/2$ and $x_{4}=1/4$. The second alternative has the lowest score $x_{2}=1/6$.  

If the scores are considered as weights to sum up $1$, we put $u=1/(1+1/6+1/2+1/4)=12/23$. Then, we have the vector $\bm{x}=(12/23,2/23,6/23,3/23 )^{T}$.  
}
\end{example}

\begin{example}
{\rm
\label{E-Anonreciprocal}
Consider a nonreciprocal matrix obtained from that in Example~\ref{E-Areciprocal} by a slight change in the entries in the form
$$
\bm{A}
=
\left(
\begin{array}{cccc}
1 & 4 & 3 & 2
\\
1/3 & 1 & 1/2 & 1/2
\\
1/4 & 2 & 1 & 3
\\
1/2 & 3 & 1/4 & 1
\end{array}
\right).
$$

To apply Theorem~\ref{T-mindAxx}, we calculate the matrices 
\begin{equation*}
\bm{A}^{-}
=
\left(
\begin{array}{cccc}
1 & 3 & 4 & 2
\\
1/4 & 1 & 1/2 & 1/3
\\
1/3 & 2 & 1 & 4
\\
1/2 & 2 & 1/3 & 1
\end{array}
\right),
\quad
\bm{B}
=
\bm{A}\oplus\bm{A}^{-}
=
\left(
\begin{array}{cccc}
1 & 4 & 4 & 2
\\
1/3 & 1 & 1/2 & 1/2
\\
1/3 & 2 & 1 & 4
\\
1/2 & 3 & 1/3 & 1
\end{array}
\right).
\end{equation*}

Evaluation of the spectral radius of the matrix $\bm{B}$ by using \eqref{E-lambda-ai1i2ai2i3aiki1} results in
$$
\mu
=
(b_{13}b_{34}b_{42}b_{21})^{1/4}
=
2.
$$

To find the matrix $\bm{B}_{\mu}^{\ast}$, we take the matrix
$$
\bm{B}_{\mu}
=
\mu^{-1}\bm{B}
=
\left(
\begin{array}{cccc}
1/2 & 2 & 2 & 1
\\
1/6 & 1/2 & 1/4 & 1/4
\\
1/6 & 1 & 1/2 & 2 
\\
1/4 & 3/2 & 1/6 & 1/2
\end{array}
\right).
$$

After calculating the matrix powers
\begin{equation*}
\bm{B}_{\mu}^{2}
=
\left(
\begin{array}{cccc}
1/3 & 2 & 1 & 4
\\
1/12 & 3/8 & 1/3 & 1/2
\\
1/2 & 3 & 1/3 & 1
\\
1/4 & 3/4 & 1/2 & 3/8
\end{array}
\right),
\qquad
\bm{B}_{\mu}^{3}
=
\left(
\begin{array}{cccc}
1 & 6 & 2/3 & 2
\\
1/8 & 3/4 & 1/6 & 2/3
\\
1/2 & 3/2 & 1 & 3/4
\\
1/8 & 9/16 & 1/2 & 1
\end{array}
\right),
\end{equation*}
we arrive at the matrix
$$
\bm{B}_{\mu}^{\ast}
=
\bm{I}\oplus\bm{B}_{\mu}\oplus\bm{B}_{\mu}^{2}\oplus\bm{B}_{\mu}^{3}
=
\left(
\begin{array}{cccc}
1 & 6 & 2 & 4
\\
1/6 & 1 & 1/3 & 2/3
\\
1/2 & 3 & 1 & 2
\\
1/4 & 3/2 & 1/2 & 1
\end{array}
\right).
$$

It is easy to see that the obtained matrix coincides with the matrix $\bm{A}_{\lambda}^{\ast}$ in Example~\ref{E-Areciprocal}. Since this matrix completely determines the score vector $\bm{x}$, the solution is the same as in Example~\ref{E-Areciprocal}. Specifically, as a score vector, one can take the vector $\bm{x}=(1,1/6,1/2,1/4)^{T}$.
}
\end{example}

\subsection{Evaluation of Scores From Several Matrices}

The problem of simultaneous approximation of several matrices naturally appears when multiple results of pairwise comparisons for the same set of alternatives according to a single criterion must be combined to produce a common score vector. 

\begin{example}
{\rm
Consider the problem of evaluating the scores on the basis of the simultaneous approximation of $m=2$ reciprocal matrices
\begin{equation*}
\bm{A}_{1}
=
\left(
\begin{array}{cccc}
1 & 3 & 4 & 2
\\
1/3 & 1 & 1/2 & 1/3
\\
1/4 & 2 & 1 & 3
\\
1/2 & 3 & 1/3 & 1
\end{array}
\right),
\qquad
\bm{A}_{2}
=
\left(
\begin{array}{cccc}
1 & 4 & 3 & 2
\\
1/4 & 1 & 1/2 & 1/2
\\
1/3 & 2 & 1 & 4
\\
1/2 & 2 & 1/4 & 1
\end{array}
\right).
\end{equation*}

To solve the problem, we apply Corollary~\ref{C-dmaxAixx}, which requires the calculation of the matrix
$$
\bm{B}
=
\bm{A}_{1}\oplus\bm{A}_{2}
=
\left(
\begin{array}{cccc}
1 & 4 & 4 & 2
\\
1/3 & 1 & 1/2 & 1/2
\\
1/3 & 2 & 1 & 4
\\
1/2 & 3 & 1/3 & 1
\end{array}
\right).
$$

The matrix $\bm{B}$ is the same as in Example~\ref{E-Anonreciprocal}. This allows the use of the results of this example, which offer the score vector in the form $\bm{x}=(1,1/6,1/2,1/4)^{T}$. 
}
\end{example}

\begin{example}
{\rm
Suppose that we have two matrices, which are nonreciprocal and given by
\begin{equation*}
\bm{A}_{1}
=
\left(
\begin{array}{cccc}
1 & 4 & 3 & 2
\\
1/3 & 1 & 1/2 & 1/2
\\
1/4 & 2 & 1 & 4
\\
1/2 & 3 & 1/4 & 1
\end{array}
\right),
\qquad
\bm{A}_{2}
=
\left(
\begin{array}{cccc}
1 & 3 & 4 & 2
\\
1/3 & 1 & 1/2 & 1/3
\\
1/3 & 2 & 1 & 3
\\
1/2 & 2 & 1/4 & 1
\end{array}
\right).
\end{equation*}

Application of Theorem~\ref{T-dmaxAixx} involves the matrices
\begin{equation*}
\bm{A}_{1}^{-}
=
\left(
\begin{array}{cccc}
1 & 3 & 4 & 2
\\
1/4 & 1 & 1/2 & 1/3
\\
1/3 & 2 & 1 & 4
\\
1/2 & 2 & 1/4 & 1
\end{array}
\right),
\qquad
\bm{A}_{2}^{-}
=
\left(
\begin{array}{cccc}
1 & 3 & 3 & 2
\\
1/3 & 1 & 1/2 & 1/2
\\
1/4 & 2 & 1 & 4
\\
1/2 & 3 & 1/3 & 1
\end{array}
\right)
\end{equation*}
to be used in the construction of the matrix
$$
\bm{B}
=
\bm{A}_{1}\oplus\bm{A}_{1}^{-}\oplus\bm{A}_{2}\oplus\bm{A}_{2}^{-}
=
\left(
\begin{array}{cccc}
1 & 4 & 4 & 2
\\
1/3 & 1 & 1/2 & 1/2
\\
1/3 & 2 & 1 & 4
\\
1/2 & 3 & 1/3 & 1
\end{array}
\right).
$$

The matrix $\bm{B}$ again coincides with the matrix in Example~\ref{E-Anonreciprocal}, which provides the same solution. 
}
\end{example}

\subsection{Weighted Scores From Several Matrices}

Suppose that there are several criteria for judging alternatives, and each criterion has a weight that indicates its relative importance among the criteria. If we have a pairwise comparison matrix obtained according to each criterion, a problem of weighted evaluation of scores arises, which is to find a single common score vector by combining the results of pairwise comparisons with the weights.

\begin{example}\label{X-A1A2A3w1w2w3}
{\rm
We examine the problem of evaluating the vector of scores given by $m=3$ reciprocal matrices
\begin{gather*}
\bm{A}_{1}
=
\left(
\begin{array}{cccc}
1 & 3 & 1 & 3
\\
1/3 & 1 & 1/4 & 1/2
\\
1 & 4 & 1 & 1/2
\\
1/3 & 2 & 2 & 1
\end{array}
\right),
\qquad
\bm{A}_{2}
=
\left(
\begin{array}{cccc}
1 & 2 & 1 & 4
\\
1/2 & 1 & 1/3 & 1/2
\\
1 & 1/3 & 1 & 1
\\
1/4 & 2 & 1 & 1
\end{array}
\right),
\\
\bm{A}_{3}
=
\left(
\begin{array}{cccc}
1 & 4 & 2 & 1/2
\\
1/4 & 1 & 1/2 & 1/3
\\
1/2 & 2 & 1 & 1/4
\\
2 & 3 & 4 & 1
\end{array}
\right),
\end{gather*}
which have to be taken with the weights
$$
w_{1}
=
1,
\qquad
w_{2}
=
1,
\qquad
w_{3}
=
1/2.
$$

To apply Corollary~\ref{C-widmaxAixx}, we calculate the matrix
$$
\bm{B}
=
w_{1}\bm{A}_{1}
\oplus
w_{2}\bm{A}_{2}
\oplus
w_{3}\bm{A}_{3}
=
\left(
\begin{array}{cccc}
1 & 3 & 1 & 4
\\
1/2 & 1 & 1/3 & 1/2
\\
1 & 4 & 1 & 1
\\
1 & 2 & 2 & 1
\end{array}
\right).
$$

Then, we find the spectral radius
$$
\mu
=
(b_{14}b_{43}b_{32}b_{21})^{1/4}
=
2,
$$
and examine the matrix
$$
\bm{B}_{\mu}
=
\mu^{-1}\bm{B}
=
\left(
\begin{array}{cccc}
1/2 & 3/2 & 1/2 & 2
\\
1/4 & 1/2 & 1/6 & 1/4
\\
1/2 & 2 & 1/2 & 1/2
\\
1/2 & 1 & 1 & 1/2
\end{array}
\right).
$$

After calculating the matrices
\begin{equation*}
\bm{B}_{\mu}^{2}
=
\left(
\begin{array}{cccc}
1 & 2 & 2 & 1
\\
1/8 & 3/8 & 1/4 & 1/2
\\
1/2 & 1 & 1/2 & 1
\\
1/2 & 2 & 1/2 & 1
\end{array}
\right),
\qquad
\bm{B}_{\mu}^{3}
=
\left(
\begin{array}{cccc}
1 & 4 & 1 & 2
\\
1/4 & 1/2 & 1/2 & 1/4
\\
1/2 & 1 & 1 & 1
\\
1/2 & 1 & 1 & 1
\end{array}
\right),
\end{equation*}
we form the matrix
$$
\bm{B}_{\mu}^{\ast}
=
\bm{I}\oplus\bm{B}_{\mu}\oplus\bm{B}_{\mu}^{2}\oplus\bm{B}_{\mu}^{3}
=
\left(
\begin{array}{cccc}
1 & 4 & 2 & 2
\\
1/4 & 1 & 1/2 & 1/2
\\
1/2 & 2 & 1 & 1
\\
1/2 & 2 & 1 & 1
\end{array}
\right).
$$

Considering that all columns in the matrix $\bm{B}_{\mu}^{\ast}$ are collinear, we take the first one to form the score vector
$$
\bm{x}
=
\left(
\begin{array}{c}
1
\\
1/4
\\
1/2
\\
1/2
\end{array}
\right)u,
\qquad
u
>
0.
$$
}
\end{example}

\subsection{Tropical Analytical Hierarchy Process}

The following example shows how the results obtained can be used to develop a tropical analog of the Analytical Hierarchy Process (AHP) decision scheme in multicriteria decision making. 
\begin{example}
{\rm
Suppose that the results of pairwise comparisons of four alternatives according to three criteria are given by the matrices $\bm{A}_{1}$, $\bm{A}_{2}$ and $\bm{A}_{3}$ from Example~\ref{X-A1A2A3w1w2w3}. Furthermore, the relative importance of the criteria is also compared, which yields the pairwise comparison matrix 
$$
\bm{C}
=
\left(
\begin{array}{ccc}
1 & 1 & 2
\\
1 & 1 & 2
\\
1/2 & 1/2 & 1
\end{array}
\right).
$$

Let us find a score vector $\bm{w}=(w_{1},w_{2},w_{3})^{T}$ from the matrix $\bm{C}$. Denote the spectral radius of the matrix $\bm{C}$ by $\nu$. Then, we obtain
$$
\nu
=
1,
\qquad
\bm{C}_{\nu}
=
\nu^{-1}\bm{C}
=
\bm{C},
\qquad
\bm{C}_{\nu}^{2}
=
\bm{C}_{\nu}
=
\bm{C}.
$$

To apply Corollary~\ref{C-mindAxx}, we calculate
$$
\bm{C}_{\nu}^{\ast}
=
\bm{I}\oplus\bm{C}_{\nu}\oplus\bm{C}_{\nu}^{2}
=
\bm{C}.
$$

Since all columns of $\bm{C}_{\nu}^{\ast}$ are collinear, we take one of them, say the first column. Finally, we have
$$
\bm{w}
=
\left(
\begin{array}{c}
1
\\
1
\\
1/2
\end{array}
\right).
$$

Note that the elements of the vector $\bm{w}$ coincide with the weights used in Example~\ref{X-A1A2A3w1w2w3}. Clearly, the solution of the multicriteria problem under consideration is the same as in this example, and can be represented as the score vector $\bm{x}=(1,1/4,1/2,1/2)^{T}$.
}
\end{example}

\section{Conclusions}

The paper aimed at reporting new developments in the area of applications of optimization techniques in tropical mathematics to pairwise comparison judgment in decision making. Various problems were considered, which arise in ranking alternatives from their pairwise comparisons. The proposed solution approach uses the approximation, in the Chebyshev or log-Chebyshev sense, of pairwise comparison matrices by consistent matrices, and reduces the problems of ranking alternatives to a tropical optimization problem.

By applying recent results in tropical optimization, we offered new direct, explicit solutions in the closed vector form, which is ready for practical implementation and further analysis. The solutions involve a finite number of simple matrix-vector operations, which offers a low polynomial computational complexity. An example was given to show that the new solution can be more accurate than other solutions previously obtained in the framework of tropical mathematics. Note that the numerical examples demonstrate low sensitivity of the solution vector to small variations in the matrix entries. A tropical analogue of the AHP decision scheme was also presented as an application example.

The development of new applications of the results to solve real-world problem of evaluating alternatives is considered as one of the main lines of future research.

\bibliographystyle{abbrvurl}

\bibliography{Using_tropical_optimization_techniques_to_evaluate_alternatives_via_pairwise_comparisons}

\end{document}